\documentclass{article}

\usepackage{amsmath}
\usepackage{amsthm}
\usepackage{amsfonts}
\usepackage{amsbsy}
\usepackage{amssymb}
\usepackage[initials]{amsrefs}
\newtheorem{theorem}{Theorem}
\newtheorem{proposition}{Proposition}

\newtheorem{lemma}{Lemma}

\newcommand{\R}{{\mathbb R}}
\newcommand{\Z}{{\mathbb Z}}

\newcommand{\set}[2]{ \left\{ #1 \ \left| \ #2 \right. \right\} }

\newcommand{\one}{\overrightarrow{1}}

\usepackage{multirow}
\usepackage{color}

\newcommand{\hess}{\mathrm{Hess~}}

\title{Damping oscillatory integrals by the Hessian determinant via Schr\"{o}dinger}
\author{Philip T. Gressman\footnote{Partially supported by NSF grant DMS-1361697 and an Alfred P. Sloan Research Fellowship.}}
\date{\today}
\begin{document}
\maketitle
\begin{abstract}
We consider the question of when it is possible to force a degenerate scalar oscillatory integral to decay as fast as a nondegenerate one by restricting the support to the region where the Hessian determinant of the phase is bounded below.  We show in two dimensions that the desired outcome is not always possible, but does occur for a broad class of phases which may be described in terms of the Newton polygon.  The estimates obtained are uniform with respect to linear perturbation of the phase and uniform in the cutoff value of the Hessian determinant. In the course of the proof, we investigate a geometrically-invariant approach to making uniform estimates of qualitatively nondegenerate oscillatory integrals. The approach illuminates a previously unknown, fundamental relationship between the asymptotics of oscillatory integrals and the Schr\"{o}dinger equation.
\end{abstract}

One of the deep results of Phong, Stein, and Sturm \cite{pss2001} in their study of uniform decay rates for oscillatory integral operators in $1+1$ dimensions is the fact that, for a real polynomial phase $\Phi$ on $\R^2$ one has the uniform estimate
\begin{equation} \left| \int_{[-1,1]^2} e^{i \lambda \Phi(x,y)} \chi \left( \epsilon^{-1} \frac{\partial^2 \Phi}{\partial x \partial y} (x,y) \right) f(x) g(y) dx dy \right| \leq C (\lambda \epsilon)^{-\frac{1}{2}} ||f||_{2} ||g||_{2} \label{pssineq} \end{equation}
for any two functions $f, g \in L^2([-1,1])$ and any $\lambda, \epsilon > 0$, where $\chi$ is a smooth function supported on the interval $[1,2]$. The constant $C$ depends only on the degree of $\Phi$ and on the cutoff $\chi$.  
In this paper, we will investigate the extent to which analogues of \eqref{pssineq} are possible in higher dimensions. Questions of stability of decay for oscillatory integrals are certainly not new in harmonic analysis, nor are approaches based on weights and damping. Readers should also see the references \cite{pss1999,karpushkin1986,ps1998,py2004,ccw1999,cw2002,seeger1994} for a variety of important formulations of and approaches to these problems. Even so, outside of those problems which are one-dimensional or exhibit some sort of one-dimensional multilinearity. Of those results which are genuinely higher-dimensional, sharp decay rates are often difficult to achieve, and damping techniques have yet to be successfully employed.

 The present paper focuses on two dimensional, translation-invariant versions of \eqref{pssineq}. By the usual $L^2$ theory, the problem reduces to that of establishing the inequality
\begin{equation} \sup_{\xi \in \R^2} \left| \int_{[-1,1]^2} e^{i (\lambda \Phi(x) + \xi \cdot x)} \chi(\epsilon^{-1} \det \hess \Phi(x)) \psi(x) dx \right| \leq C \lambda^{-1} \epsilon^{-\frac{1}{2}}, \label{vdcq} \end{equation}
where, for convenience, an additional smooth cutoff function $\psi$ has been added with support in some small neighborhood of the origin. (The notation $\hess \Phi$ is reserved for the usual Hessian of $\Phi$ and $\nabla^2 \Phi$ for an intrinsic geometric version of the Hessian to appear later on.)

Unlike the robust sense in which \eqref{pssineq} holds, the inequality \eqref{vdcq} can fail under fairly routine circumstances.
If $\Phi(x_1,x_2) := (x_2 + x_1^2)^2$, one sees that the Hessian determinant equals $8(x_2+x_1^2)$. 
To evaluate the left-hand side of \eqref{vdcq} at $\xi = 0$ in the special case $\epsilon = \lambda^{-\frac{1}{2}}$, make the change of variables $x_2 \mapsto \lambda^{-1/2} x_2 - x_1^2$; the integral can easily be seen to be asymptotic to
\[ \lambda^{-\frac{1}{2}} \left( \int e^{i x_2^2} \chi(8x_2) dx_2 \right) \left( \int \psi(x_1,-x_1^2) dx \right) \]
as $\lambda \rightarrow \infty$. In particular, the coefficient of $\lambda^{-1/2}$ will typically be nonzero.  However, if the full inequality \eqref{vdcq} held, the decay would have been $\lambda^{-3/4}$.

Even so, \eqref{vdcq} is generally true if not universally.
Suppose that $\Phi$ is real analytic on a neighborhood of the origin in $\R^2$ and that $\Phi(0,0) = \nabla \Phi(0,0) = 0$, and let $\Gamma$ be the Newton polygon associated to $\Phi$ at the origin, i.e., $\Gamma \subset [0,\infty)^2$ is the convex hull of the union of all quadrants $[k_1,\infty) \times [k_2, \infty)$ such that $$\frac{\partial^{k_1+k_2} \Phi}{\partial^{k_1} x_1 \partial^{k_2} x_2}(0,0) \neq 0.$$
To each compact edge $e$ of $\Gamma$, we associate a polynomial $\phi_e$ by restricting the Taylor series of $\Phi$ at the origin to those terms lying on $e$, i.e.,
\[ \phi_e (x) := \sum_{(k_1,k_2) \in e} \frac{x_1^{k_1} x_2^{k_2}}{k_1! k_2!} \frac{\partial^{k_1+k_2} \Phi}{\partial^{k_1} x_1 \partial^{k_2} x_2}(0,0). \]
For each edge polynomial $\phi_e$, it will be assumed that the mapping $(x_1, x_2) \mapsto \nabla \phi_e(x_1,x_2) $ has at at worst Whitney fold singularities away from the $x_1$- and $x_2$-axes. If the edge $e$ happens to meet the first (horizontal) axis in $[0,\infty)^2$, the mapping $x \mapsto \nabla \phi_e(x)$ will be required to have no worse than a Whitney fold singularity on the $x_2$-axis away from the origin, and likewise for edges meeting the second (vertical axis) and singularities on the $x_1$-axis away from the origin.  By a Whitney fold singularity of a mapping between manifolds of the same dimension, we mean that when the mapping has a degenerate differential at a point, it is degenerate on a hypersurface passing through any such point, that the determinant of the differential vanishes only to first order near that hypersurface, and that the differential is injective when restricted to the hypersurface. The reader will note that this condition is slightly stronger than the usual condition of Var\v{c}enko \cite{varcenko1976}, but the extra strength seems to be necessary as the local behavior of both $\Phi$ and $\det \hess \Phi$ must be simultaneously described by the structure of this single Newton polygon.

The main result of this paper is as follows:
\begin{theorem}
When $\Phi$ is as above, there is a neighborhood $U$ of the origin in $\R^2$ such that, for every smooth $\psi$ supported on $U$ and every smooth $\chi$ supported on $[-2,-1] \cup [2,1]$, there is a finite $C$ such that
\begin{equation} \sup_{\xi \in \R^2} \left| \int e^{i \lambda ( \Phi(x) + \xi \cdot x)} \chi( \epsilon^{-1} \det \hess \Phi(x)) \psi(x) dx \right| \leq C \lambda^{-1} \epsilon^{-\frac{1}{2}} \log^s \frac{1}{\epsilon} \label{mainest} \end{equation} for all positive $\lambda$ and $\epsilon$. The exponent $s$ equals zero unless the Newton polygon $\Gamma$ meets the diagonal at a vertex, in which case $s=1$. \label{mainthm}
\end{theorem}

The proof of Theorem \ref{mainthm} proceeds in two stages. The first takes place in Section \ref{geomsec}, in which uniform, geometrically-invariant estimates for {nondegenerate} scalar oscillatory integrals are investigated. The challenge is that, while any critical points of $\Phi + \xi \cdot x$ on the support of the cutoff are technically nondegenerate, there is essentially no quantitative control of the nondegeneracy. Consequently, there is need for uniform estimates which hold for all $\lambda$ rather than just being true as $\lambda \rightarrow \infty$.  The remarkable thing about this analysis is that it identifies a previously-unknown, fundamental relationship between the asymptotic expansion of any nondegenerate oscillatory integral and the formal expansion of certain solution operators of Schr\"{o}dinger-type equations. The basis for this connection is found in Lemma \ref{bigschrodlem}, which establishes that any smooth phase $\Phi$ with a nondegenerate critical point at $p$ in a manifold with torsion-free connection $\nabla$ intrinsically generates (i.e., without reference to any choice of coordinates) a second-order operator $\square$ in a neighborhood of $p$ for which 
\[ \left( \frac{\partial}{\partial t} - \frac{i}{2} \square \right) \left[ t^{-\frac{n}{2}} e^{i t^{-1} (\Phi - \Phi(p))} \right] = 0 \]
holds for all $t > 0$. This realization in turn yields a much deeper understanding of the geometric significance of the coefficients found when asymptotically expanding nondegenerate oscillatory integrals (see \eqref{asympset}), namely that, up to a constant, the coefficients of the asymptotic expansion match the Taylor coefficients of the formal expansion of $\exp(i\square^* t/2)$ applied to the amplitude $\psi$ (here $\square^*$ is an appropriate adjoint of $\square$ since it need not be self-adjoint). For this reason, the estimates derived in Section \ref{geomsec} are presented in a general way without reference to the specific context of Theorem \ref{mainthm} to facilitate future application to other problems. 

The second stage of the proof comes in Section \ref{mainproof}, in which the results of Section \ref{geomsec} are applied to phases $\Phi$ satisfying the hypotheses of Theorem \ref{mainthm}.  The main argument of this section is built on the familiar analysis of boxes in the bi-dyadic decomposition of the plane.

\section{Geometric Oscillatory Integral Analysis}
\label{geomsec}

It is certainly the case that the {asymptotic} behavior of oscillatory integrals with nondegenerate critical points is well-understood: if $\Phi$ is a $C^\infty$ function on some neighborhood of $x_0 \in \R^n$ and $x_0$ is a nondegenerate critical point of $\Phi$, then
\begin{equation} \int e^{i \lambda \Phi(x)} \psi(x) dx \sim \sum_{j=0}^\infty a_j \lambda^{-\frac{n}{2}-j} \mbox{ as } \lambda \rightarrow +\infty \label{asym} \end{equation}
when $\psi \in C^\infty(\R^n)$ is supported on a sufficiently small neighborhood of $x_0$. Each  $a_j$ depends only on finitely many derivatives of $\Phi$ and $\psi$ at $x_0$.  For example, 
\begin{equation} a_0 =  \frac{\pi^{\frac{n}{2}} e^{i \frac{\pi}{4} \omega} \psi(x_0)}{\sqrt{ |\det \hess \Phi(x_0)|}} \label{zeroord} \end{equation}
where $\omega$ equals the number of positive eigenvalues of $\hess \Phi (x_0)$ minus the number of negative eigenvalues.
H\"{o}rmander \cite{hormanderI} and Stein \cite{steinha}, for example, give somewhat distinct proofs of these and other facts.  

Despite the wealth of knowledge, there are several important reasons to revisit the asymptotics \eqref{asym}. One is the need to employ a fully coordinate-independent approach that only exploits the {\it relative} geometry of the phase $\Phi$ and the amplitude $\psi$.  In order to {pose} a scalar oscillatory integral problem, the only structures intrinsically involved are an $n$-dimensional manifold ${\mathcal M}$, a measure $\mu$ of smooth density on ${\mathcal M}$, a real phase $\Phi$, and an amplitude $\psi$. With this minimal structure, the integral
\begin{equation} \int e^{ i \lambda \Phi} \psi d \mu \label{geomint} \end{equation}
makes sense and can be studied asymptotically or uniformly.  This formulation of the  problem may legitimately be considered qualitative, since, e.g., no information about magnitudes of derivatives of $\Phi$ or $\psi$ is available.  Instead, the challenge is to discover useful structures rather than imposing them.

Although the minimal structure of \eqref{geomint} is sufficient for asymptotics, meaningful uniform estimates do not seem possible without some additional assistance.  In the case of Theorem \ref{mainest}, the manifold ${\mathcal M}$ happens to come equipped with a torsion-free connection $\nabla_{\cdot}$: on $\R^n$, the standard connection is preserved by affine transformations, unlike metrics. The connection, though still essentially qualitative, imposes just enough additional structure for uniform estimates.

For readers uninspired by coordinate-independent methods, there is still another reason to reexamine this ostensibly simple statement \eqref{asym}.  In order to prove Theorem \ref{mainthm}, it is necessary to establish uniform estimates which are manifestly stable with respect to perturbations of the phase and the amplitude.  Because the Hessian determinant will be restricted to the region where it is comparable to $\epsilon$, special care must be taken when understanding the sort of perturbations which arise, as uniformity in $\epsilon$ can only be accomplished through a precise understanding of the terms (and, in particular, the magnitude of the derivatives) that will appear.  Along the way, it will be necessary to establish stability of the higher-order terms error terms of the asymptotic expansion as well.

\subsection{Oscillatory integrals and Schr\"{o}dinger equations}

The first step is a proof of Lemma \ref{bigschrodlem}, which identifies a fundamental relationship between oscillatory integral asymptotics and pseudo-Riemannian Schr\"{o}dinger equations. Specifically, it identifies every nondegenerate oscillatory integral as a sort of point fundamental solution of a Schr\"{o}dinger equation:
\begin{lemma}
Suppose $\Phi$ is a smooth, real-valued function on some open subset $U$ of am $n$-dimensional manifold ${\mathcal M}$ with torsion-free connection $\nabla$ and that $\Phi$ has a nondegenerate critical point at some point $p \in U$. Then there is a neighborhood $\Omega$ of $p$ and a smooth,  second-order differential operator $\square$ on $\Omega$  such that \label{bigschrodlem}
\begin{equation} \left( \frac{\partial}{\partial t} -  \frac{i}{2} \square \right) \left[ t^{-\frac{n}{2}} e^{i t^{-1} (\Phi-\Phi(p))} \right] = 0 \label{schrodtime} \end{equation}
for all $t > 0$. Moreover, for all sufficiently small smooth perturbations of $\Phi$ and $\nabla$, on the same set $\Omega$, corresponding operators $\square$ continue to exist so that the perturbed version of \eqref{schrodtime} (where $p$ is understood to always be the critical point of the pertrubed $\Phi$) still holds. The operator $\square$ depends only on $\Phi$ and $\nabla$ (i.e., no coordinate choices are necessary) and varies smoothly as they vary.
\end{lemma}

The first step in the proof is to observe that a geometrically-invariant Hessian of $\Phi$ is well-defined given a torsion-free connection (i.e., one does not need a metric).
Suppose $\nabla$ is such a connection and $\Phi$ is a real-valued function, both smooth and both defined on some open subset $U$ of the manifold ${\mathcal M}$.   We define the Hessian $\nabla^2 \Phi$ to be the quadratic form on vectors given by
\begin{equation} \nabla^2 \Phi (X,Y) := Y X \Phi - (\nabla_Y X) \Phi. \label{classichessian} \end{equation}
Clearly \eqref{classichessian} only depends on pointwise values of the vector field $Y$, and because $\nabla$ is torsion-free, $\nabla^2 \Phi$ is symmetric; thus $\nabla^2 \Phi$ is a well-defined tensor field.  It is a trivial but important fact that  $ \nabla^2 \Phi$ is independent of the connection at every critical point of $\Phi$ because the term $(\nabla_Y X) \Phi$  vanishes there.

Given a critical point $p$ of $\Phi$, let $\Omega_0$ and $\Omega_1$ be open neighborhoods of $p$ such that the closure of $\Omega_0$ is compact and contained in $\Omega_1$ and any two points $p', q'$ in the closure of $\Omega_0$ are connected by a unique geodesic curve $\gamma \subset \Omega_1$ (meaning $\nabla_{\dot \gamma} \dot \gamma = 0$, $\gamma(0) = p'$, and $\gamma(1) = q'$ as usual).  Note that this geodesic connectivity property of $\Omega_0$ and $\Omega_1$ will continue to be true for any sufficiently small perturbation of the connection $\nabla$.  It's also true that one may choose $\Omega_1$ sufficiently small so that $\Phi$ has a unique critical point in $\Omega_1$ at the point $p$.  The existence of a unique critical point in $\Omega_0$ and no critical points in $\Omega_1 \setminus \Omega_0$ will also continue to be true for sufficiently small perturbations of $\Phi$.
For any real $\sigma > 0$ and any $q \in \Omega_0$, we define a weighted Hessian $\nabla^2_{\! \! \sigma} \Phi$ at $q$ by taking $\gamma$ to be the geodesic with $\gamma(0) = p$, $\gamma(1) = q$, and setting
\begin{align}
\nabla^2_{\! \! \sigma} \Phi (X,Y) & := \sigma \int_0^1 (1-s)^{\sigma-1} \left.  \nabla^2 \Phi \right|_{\gamma(s)} \left( \left. X \right|_{\gamma(s)}, \left. Y \right|_{\gamma(s)} \right) ds, \label{weightedhessian}
 \end{align}
where we extend $X$ and $Y$ by reverse parallel transport along $\gamma$ so that $X_{\gamma(1)} = X$ and likewise for $Y$.  We are primarily interested in the cases $\sigma = 1,2$, but it is perhaps worth remarking that the usual Hessian \eqref{classichessian} is the limit of $\nabla^2_{\! \! \sigma} \Phi$ as $\sigma \rightarrow 0^{+}$.  Regardless of the choice of $\sigma$, all weighted Hessians agree at the critical point $p$.  
Because the geodesics depend smoothly on $p$ and the connection $\nabla$, the weighted Hessians will also vary smoothly when $\Phi$ and $\nabla$ are perturbed.  Let us also restrict $\Omega_0$ as necessary so that $\nabla_{\! \! 1}^2 \Phi$ is nondegenerate on the closure of $\Omega_0$. Such nondegeneracy will consequently also be true for small perturbations of $\Phi$ and $\nabla$.  The neighborhood $\Omega$ from Lemma \ref{bigschrodlem} can now be chosen to equal any open subset whose closure is contained in $\Omega_0$.  The operator $\square$ is identified and the proof of Lemma \ref{bigschrodlem} is completed via the next proposition.
\begin{proposition}
Given  $\psi \in C^\infty(\Omega_0)$, define the vector field $Z_{\psi}$ by \label{schrodlem}
\begin{align}
Z_{\psi} f &  :=  (\nabla_{\! \! 1}^2 \Phi)^{-1} \left( d \psi, d f \right) \label{vf1} 
\end{align}
for all $f$ on $\Omega_0$.  
Also consider the operator $\square_0$ and function $\eta$ on $\Omega_0$ given by
\begin{align} 
\square_0 f & :=  \mathrm{tr} ( (\nabla_{\! \! 1}^2 \Phi)^{-1} \nabla_{\! \! 2}^2 \Phi (\nabla_{\! \! 1}^2 \Phi)^{-1} \nabla^2 f), \label{schrod0} \\
\eta(q)  &  := \int_0^1 \left[ n - \square_0 \Phi ( \gamma(s)) \right] \frac{ds}{s} \label{etafn}, \end{align}
where $\gamma$ is the geodesic with $\gamma(0) = p$ and $\gamma(1) = q$.  Then the operator 
\begin{equation} \square f := \square_0 f + Z_\eta f \label{boxop} \end{equation}
satisfies \eqref{schrodtime} and varies smoothly under small perturbations of $\Phi$ and $\nabla$.
\end{proposition}

\begin{proof}
Without loss of generality, we may assume $\Phi(p) = 0$.  The smoothness of $Z_\psi$ is straightforward. Smoothness of $\eta$ follows once it is observed that $\square_0 \Phi$ must equal $n$ at the point $p$, because $\mathrm{tr} ( (\nabla_{\! \! 1}^2 \Phi)^{-1} \nabla_{\! \! 2}^2 \Phi (\nabla_{\! \! 1}^2 \Phi)^{-1} \nabla^2 \Phi)$ is the trace of the identity  at $p$ since $\nabla^2 \Phi = \nabla_{\! \! 1}^2 \Phi = \nabla_{\! \! 2}^2 \Phi$ at $p$.  

Let $D$ be the vector field on $\Omega_0$ such that, at $q$, $D = \dot \gamma(1)$ for the geodesic $\gamma$ with $\gamma(0) = p$ and $\gamma(1) = q$. We then have the identities
\begin{align}
\left. \nabla^2_{\! \! 1} \Phi \right|_{q} (D, X)   = \left. X \Phi \right|_{q}, \label{hessid} \qquad
\left. \nabla^2_{\! \! 2} \Phi  \right|_{q} (D,D)   = 2 \Phi(q),
\end{align}
for any vector $X$.
Both formulas are a consequence of integration by parts.  More specifically, one first expands definition \eqref{weightedhessian} using \eqref{classichessian} to conclude that
\begin{align*}
\left. \nabla_{\! \! \sigma}^2 \Phi  \right|_{q} (D, X) 
& = \sigma \int_0^1 (1-s)^{\sigma-1} \left[ \dot \gamma(s) (X \Phi) -(\nabla_{\dot \gamma} X) \Phi \right]_{\gamma(s)} ds.
\end{align*}
Now $(\nabla_{\dot \gamma} X) \Phi = 0$ because $X$ has been parallel transported. The remaining term satisfies $\dot \gamma(s) (X \Phi) = \frac{d}{ds} \left[  X \Phi (\gamma(s)) \right]$.  We integrate by parts to see that the first identity of \eqref{hessid} must hold (using the fact that the boundary term at $s=0$ vanishes because $\Phi$ has a critical point). Doing the same integration by parts with $\sigma=2$ and taking $X = D$ gives instead that
\[ \left. \nabla^2_{\! \! 2} \Phi  \right|_{q} (D,D) = 2 \int_0^1 (\dot \gamma \Phi)(\gamma(s)) ds. \]
A second integration by parts in this formula gives the second part of \eqref{hessid}.

Since $\nabla_{\! \! 1}^2 \Phi$ is nondegenerate on $\Omega_0$, we may write first identity of \eqref{hessid} as  $(\nabla_{\! \! 1}^2 \Phi)^{-1} d \Phi = D$ (since $\nabla_{\! \! 1}^{2} \Phi ( D, \cdot) = d \Phi ( \cdot)$ as linear functionals on vectors).  If we apply $\square_0$ to $t^{-\frac{n}{2}} e^{i t^{-1} \Phi}$, we get that
\begin{align*} \square_0 \left( t^{-\frac{n}{2}} e^{i t^{-1} \Phi} \right) & = \left( - i t^{-1} \square_0 \Phi - t^{-2} (\nabla_{\! \! 1}^2 \Phi)^{-1} \nabla_{\! \! 2}^2 \Phi (\nabla_{\! \! 1}^2 \Phi)^{-1} (d \Phi, d \Phi) \right)  t^{-\frac{n}{2}} e^{i t^{-1} \Phi} \\
& = \left( - i t^{-1} \square_0 \Phi - 2 t^{-2} \Phi  \right)  t^{-\frac{n}{2}} e^{i t^{-1} \Phi}.
\end{align*}
Since $Z_\eta \Phi = Z_\Phi \eta = D \eta$, the identity \eqref{schrodtime} will hold if and only if $D \eta = n - \square_0 \Phi$.
Along any geodesic beginning at $p$, this equation becomes
\[ s \frac{d}{ds} \eta ( \gamma(s)) = n - \square_0 \Phi ( \gamma(s)). \]
Because $n - \square_0 \Phi$ vanishes at $p$, this equation has a unique smooth solution for which $\eta$ vanishes at $p$ as well, i.e., \eqref{etafn}.
\end{proof}

\subsection{Pointwise Schr\"{o}dinger Estimates}

The next lemma establishes explicit asymptotics with remainder terms for nondegenerate oscillatory integrals in terms of the operator $\square$ constructed in the previous section. The advance of this lemma over the variants from \cite{hormanderI} or \cite{steinha} is that it's now relatively easy to describe the coefficients in the asymptotic expansion in terms of the formal Taylor series expansion of the operator $\exp (i t \square^* / 2)$ at $t=0$, where $\square^*$ is an appropriate adjoint of $\square$ from \eqref{boxop}.  Specifically, the reader may see the connection between \eqref{asympset} and H\"{o}rmander's Theorem 7.7.5 \cite{hormanderI}, but the decay of the error term in H\"{o}rmander's result is not sharp (i.e., it requires more regularity than is necessary).  Instead, \eqref{asympset} is closer to H\"{o}rmander's Lemma 7.7.3, which applies explicitly to quadratic phases; although one could change variables {\it a la} Stein \cite{steinha} to apply this lemma more generally, this method would still yield a coordinate-dependent estimate for the error term.  

\begin{lemma}
Let $\Phi$ and $\Omega$ be the operator and open set identified in Lemma \ref{schrodlem} and consider the function \label{schrodest}
\[ I_\psi (t) := \int_{\Omega} t^{-\frac{n}{2}} e^{i t^{-1} \Phi} \psi~d \mu \]
for $\psi$ smooth and compactly supported on $\Omega$ and $d \mu$ some measure generated by a smooth, nonvanishing density $\mu$ on ${\mathcal M}$. Let $\omega$ equal the number of positive eigenvalues of $\nabla^2 \Phi(p)$ minus the number of negative eigenvalues. Then as $t \rightarrow 0^+$, the difference
\begin{equation} I_\psi(t) - \pi^{\frac{n}{2}} e^{i \frac{\pi}{4} \omega} e^{i t^{-1} \Phi(p)}  \left( \frac{\sqrt{|\det \nabla^2 \Phi|}}{\mu}(p) \right)^{-1} \sum_{\ell=0}^N \frac{(i {\square^*)}^\ell \psi(p)}{2^{\ell} \ell!} t^\ell \label{asympset} \end{equation}
is $O(t^{N+1})$ as $t \rightarrow 0^+$ for each $N$.  The operator $\square^*$ is the adjoint of \eqref{boxop} with respect to $d \mu$.  If the magnitude of the difference \eqref{asympset} is denoted $E_{\psi}^N(t)$ and if $k$ is any integer strictly greater than $\frac{n}{2}$, then
\begin{align}
E_\psi^N(t) & \lesssim t^{N+1}  \left( \int | (\square^*)^{N+1} \psi| d \mu \right)^{1 - \frac{n}{2k}} \left( \int | (\square^*)^{N+k+1} \psi| d \mu \right)^{\frac{n}{2k}}, \label{error1}
\end{align}
where the implicit constant depends only on $N$, $n$, and $k$.
\end{lemma}
\begin{proof}
We begin by establishing the following fact about functions on the real line: Suppose $f \in C^k(\R_{> 0})$ and satisfies the bounds $|f(t)| \leq A t^{-n/2}$ and $|f^{(k)}(t)| \leq B t^{-n/2}$ on $(0,\infty)$ for some positive $n < 2k$. Then $f$ is bounded on $(0,\infty)$ and $||f||_\infty \lesssim A^{(2k-n)/(2k)} B^{n/(2k)}$ with an implicit constant that depends only on $n$ and $k$.  
To see this, let $\eta$ be any $C^\infty$ function which is identically one on $(-\infty,1]$ and identically zero on $[2,\infty)$; for any positive integer $k$, let
\[ \omega_k(t) := \frac{t^{k-1}}{(k-1)!} \eta ( t) \chi_{[0,\infty)} (t) \mbox{ and } \tilde \omega_k(t) := \sum_{l=1}^k \binom{k}{l} \eta^{(\ell)} ( t) \frac{ t^{l-1}}{(l-1)!}. \]
The $k$-th derivative of $\omega_k$ in the sense of distributions is exactly $\tilde \omega_k$ plus a Dirac delta function at $t=0$. Consequently, for any $\rho > 0$ and any $t_0 \in \R$, we have
\[ (-1)^k \rho^k \int f^{(k)}(t+t_0) \rho^{-1} \omega_k( \rho^{-1} t) dt = f(t_0) + \int f(t+t_0) \rho^{-1} \tilde \omega_k (\rho^{-1} t) dt \]
for any function $f$ which is $C^k$ on some interval containing $[t_0,t_0+2 \rho]$.  Now if $|f(t)| \leq A t^{-n/2}$ on the positive real line, then
\[ \int \left| f(t+t_0) \rho^{-1} \tilde \omega_k (\rho^{-1} t) \right| dt  \leq A \int_\rho^{2 \rho} t^{-\frac{n}{2}} \rho^{-1} |\tilde \omega_k( \rho^{-1} t)| dt \leq A \rho^{-\frac{n}{2}} ||\tilde \omega_k||_1 . \]
Likewise, if $|f^{(k)}(t)| \leq B t^{-n/2}$, then
\begin{equation}  \rho^k \int \left| f^{(k)}(t+t_0) \rho^{-1} \omega_k( \rho^{-1} t) \right| dt \leq B \rho^{k - \frac{n}{2}} \int_0^2 t^{-\frac{n}{2}} |\omega_k(t)| dt. \label{temp1} \end{equation}
As long as $k > \frac{n}{2}$, the integral on the right-hand side of \eqref{temp1} will be finite.  Consequently, for any $t_0 > 0$, we will have
\[ |f(t_0)| \lesssim A \rho^{-\frac{n}{2}} + B \rho^{k - \frac{n}{2}}. \]
This inequality can be optimized over $\rho$ by taking $\rho$ comparable to $(A/B)^{1/k}$ (one may assume $B \neq 0$ since if $f$ is a polynomial then the bound $|f(t)| \leq A t^{-\frac{n}{2}}$ would imply $f$ vanishes identically). This leads to the conclusion that
\[ |f(t_0)| \lesssim A^{1- \frac{n}{2k}} B^{\frac{n}{2k}} \]
uniformly for all $t_0 > 0$, where the implicit constant depends only on $n$ and $k$.

Now we apply this observation to $f(t) := I_\psi(t)$.  Without loss of generality, assume that $\Phi(p) = 0$ at the critical point $p$.  By \eqref{schrodtime} we have that
\begin{equation} \frac{d^k}{dt^k} I_{\psi} (t) = \left( \frac{i}{2} \right)^k I_{(\square^*)^k \psi}(t) \label{differr} \end{equation}
and so we must have
\[ |I_\psi(t)| \leq t^{-\frac{n}{2}} \int |\psi| d \mu \mbox{ and } |I^{(k)}_\psi(t)| \leq 2^{-k} t^{-\frac{n}{2}} \int \left| \left( \square^*\right)^k \psi \right| d \mu. \]
Consequently the bound for $f$ above gives that
\begin{equation} |I_\psi(t)| \lesssim \left( \int \left| \left( \square^*\right)^k \psi \right| d \mu \right)^{\frac{n}{2k}} \left( \int |\psi| d \mu \right)^{1 - \frac{n}{2k}} \label{zeroest} \end{equation}
for any $k > \frac{n}{2}$ with a constant depending only on $n$ and $k$.

To compute higher-order errors, we go back to the Taylor polynomial for $I_{\psi}(t)$.  From the usual calculation \eqref{zeroord}, we know that $I_{(\square^*)^\ell \psi}(t_0)$ tends to finite limit as $t_0 \rightarrow 0^+$, namely to 
\begin{equation} \frac{\pi^{\frac{n}{2}} e^{i \frac{\pi}{4} \omega}  \mu(p) (\square^*)^\ell \psi(p)}{\sqrt{|\det \nabla^2 \Phi(p)|}} . \label{newzeroord} \end{equation}
Thus the sum appearing in \eqref{asympset} is simply the degree $N$ Taylor polynomial of $I_\psi(t)$ at $t=0$. By the remainder formula for Taylor's theorem, we know that the error $E^N_\psi(t)$ will be bounded in magnitude by $t^{N+1}$ times the supremum of the magnitude of the $(N+1)$-st derivative of $I_\psi(t)$. Combining \eqref{differr} and \eqref{zeroest} gives exactly \eqref{error1}.
\end{proof}

\subsection{Nondegenerate cutoffs}

\label{nondegw}

In addition to Lemma \ref{schrodest}, it will also be necessary to establish an essentially lower-dimensional estimate with respect to a foliation of $\Omega$ under the assumption that the amplitude $\psi$ is not necessarily smooth when examined transversely to the leaves. The net effect is that the nondegenerate operator \eqref{boxop} is replaced by a degenerate variant which takes into account the geometry of the foliation. At first glance, it appears as though there is a cost in terms of decay in the asymptotic expansion, but in reality this is not the case.

For convenience, we will take the foliation to be level sets of a smooth function $u$ with nonvanishing differential. Specifically, any smooth function $u$ will be called a nondegenerate cutoff function for the phase $\Phi$ at the point $p$ when the following criteria are satisfied: 
\begin{itemize}
\item At the point $p$,  $du \neq 0$.
\item The phase $\Phi$ has a critical point at $p$ when restricted to the hypersurface of constant $u$ passing through $p$. That is, $d \Phi \wedge d u = 0$ and $d \Phi = c d u$ for some constant $c$. For obvious reasons, this constant will be called $\frac{d \Phi}{d u}(p)$.
\item The critical point of $\Phi - \frac{d \Phi}{du}(p)  u$ at $p$ is nondegenerate when restricted to the hypersurface of constant $u$ passing through $p$.  That is, $\nabla^2 \Phi - \frac{d\Phi}{du}(p) \nabla^2 u$ restricts to a nondegenerate quadratic form on vectors to the hypersurface of constant $u$ at $p$. (Note that since $\Phi - \frac{d \Phi}{du}(p) u$ has a critical point,  $\nabla^2 \Phi - \frac{d \Phi}{du}(p) \nabla^2 u$ is independent of the connection at $p$.)

\end{itemize}

Near $p$, the set of points at which $d \Phi$ and $du$ are linearly dependent (i.e., $d \Phi \wedge du = 0$) will be a curve. To see this, simply define linearly-independent vector fields $X_1,\ldots,X_{n-1}$ near $p$ for which $X_j u = 0$. Now $d \Phi \wedge d u = 0$ holds exactly when $X_j \Phi = 0$ for $j=1,\ldots,n-1$. Since the critical point of $\Phi - \frac{d\Phi}{du}(p) u$ on the hypersurface of constant $u$ is assumed to be nondegenerate, we know that the matrix $X_k  X_j (\Phi - \frac{d \Phi}{du}(p) u)$ will have full rank at $p$. But $X_k X_j u = 0$ since the vector fields are tangent to level hypersurfaces of $u$.  Consequently $X_k X_j \Phi$ will also be a matrix of rank $(n-1)$, so the implicit function theorem guarantees that there is a curve $\gamma$ which is transverse to the level hypersurfaces of $u$ which will parametrize the zero set of $d \Phi \wedge d u$ in a neighborhood of $p$.  At all such points $p'$ sufficiently near $p$ at which $d \Phi \wedge du = 0$, we may further assume that $\nabla^2 \phi - \frac{d \Phi}{du}(p') \nabla^2 u$ is nondegenerate on its own level hypersurface of $u$.

Let us return to the analogue of Lemma \ref{bigschrodlem} when $\Phi$ is accompanied by a nondegenerate cutoff function $u$ at the point $p$.  We know that there exists an open interval $I$ and a curve $\gamma : I \rightarrow {\mathcal M}$ such that $\gamma$ parametrizes the zero set of $d \Phi \wedge du$ near $p$. When we parametrize $\gamma$ so that $u(\gamma(s)) = s$, we have that this parametrization varies smoothly in $u$ and $\Phi$ and the interval $I$ can be taken to be independent of sufficiently small perturbations. 
We may assume that $du$ does not vanish on the closure of $\gamma$, so we can find small neighborhoods in the hypersurfaces of constant $u$ passing through the curve $\gamma$ on which we may apply Lemma \ref{bigschrodlem} with the connection $\nabla_\cdot$ {\it on the hypersurfaces} of constant $u$ being chosen to equal the Levi-Civita connection of $\nabla^2 \Phi - \frac{d \Phi}{du} \nabla^2 u$ restricted to the hypersurfaces of constant $u$.  We conclude that there exists a neighborhood $\Omega$ of $p$ and a differential operator $\square$ on the closure of that neighborhood such that $\square$ commutes with multiplication by any function of $u$ and
\begin{equation}
\left( \frac{\partial}{\partial t} - \frac{i}{2} \square \right) \left[ t^{- \frac{n-1}{2}} e^{ i t^{-1} ( \Phi - f(u))} \right] = 0 \label{schrodtime2}
\end{equation}
where $f$ is the function given by $f(s) := \Phi(\gamma(s))$.  (Note that $f'(s) = d \Phi(\dot \gamma) = \frac{d \Phi}{du}(\gamma(s)) du( \dot \gamma)$, so $f'(s) = \frac{d \Phi}{du}(\gamma(s))$ when we parametrize $\gamma$ as already specified.) As always, we can assume that this neighborhood is constant for sufficiently small perturbations of $\Phi$, $u$, and the original connection $\nabla_\cdot$, and we know that $\square$ and $f$ vary smoothly under such perturbations.  We can (and must) also choose $\Omega$ so that $\nabla^2 \Phi - \frac{d \Phi}{du} \nabla^2 u$ is nondegenerate on level hypersurfaces of $u$ at every point on the closure of $\gamma$.


Once the Schr\"{o}dinger equation \eqref{schrodtime2} is known, one may develop an immediate analogue of the asymptotic expansion \eqref{asympset}. The proof is exactly the same as the proof of Lemma \ref{schrodest} except for one point, which is the evaluation of the limit
\[ \lim_{t \rightarrow 0^+} t^{-\frac{n-1}{2}} \int_{U} e^{i t^{-1} ( \Phi - f(u))} \psi d \mu. \]
This limit must certainly exist, since by Fubini we may express the integral as an integral over submanifolds of constant $u$.  In particular, the limit will be an integral of $\psi$ over the curve $\gamma$, since this curve parametrizes the critical points of the phase $\Phi - f(u)$ on the constant-$u$ submanifolds. Using stationary phase, we conclude that
\begin{equation} \lim_{t \rightarrow 0^+} t^{-\frac{n-1}{2}} \int_{U} e^{i t^{-1} ( \Phi - f(u))} \psi d \mu =  \pi^{\frac{n-1}{2}} e^{i \frac{\pi}{4} \omega} \int  \psi(\gamma(s)) d \varphi(s) \label{curveeq} \end{equation}
where $\omega$ is the number of positive eigenvalues minus negative eigenvalues when the Hessian $\nabla^2 \Phi - f'(u) \nabla^2 u$ is restricted to submanifolds of constant $u$.  Here we used the fact that the Hessian $\nabla^2 ( \Phi - f(u))$ and $\nabla^2 \Phi - f'(u) \nabla^2 u$ agree on hypersurfaces of constant $u$. The density $\varphi$ can be determined by the coarea formula: by Fubini, we may write
\[ \int  f d \mu = \int \left( \int_{u = s} f \left. d \mu \right|_{u=s} \right) ds \]
where $\left. \mu \right|_{u = s}$ is a density on the hypersurface $u=s$; if $X_1,\ldots,X_{n-1}$ are linearly-independent vectors tangent to $u=s$, then 
\[ \left. \mu \right|_{u=s}(X_1,\ldots,X_{n-1}) = \mu(V,X_1,\ldots,X_{n-1}) \]
for any vector field $V$ with $V u = 1$. Evaluating the limit \eqref{newzeroord} on the hypersurface $u = s$ gives $d \varphi = \frac{d \varphi}{ds } ds$, where
\begin{equation} \frac{d \varphi}{d s} = \frac{ \mu(\dot \gamma,X_1,\ldots,X_{n-1})}{\sqrt{|\det (\nabla^2 \Phi - f'(u) \nabla^2 u) (X_j, X_k)_{jk}|}} \label{rnd} \end{equation}
where $X_1,\ldots,X_{n-1}$ are any choice of linearly-independent tangent vectors to the hypersurface $u=s$ and in the denominator we mean the determinant of the $(n-1)\times(n-1)$ matrix whose $jk$-entry equals $(\nabla^2 \Phi - f'(u) \nabla^2 u) (X_j, X_k)_{jk}$. The vector $\dot \gamma$ appears because $\dot \gamma u =1$, and the Hessian determinant in the denominator comes from \eqref{newzeroord} combined with the fact that $\nabla^2 \Phi - f'(u) \nabla^2 u$ agrees with $\nabla^2 (\Phi - f(u))$ along these hypersurfaces.
Following Lemma \ref{schrodest} exactly, we conclude that the difference
\begin{equation} t^{- \frac{n-1}{2}} \int e^{i t^{-1} ( \Phi - f(u))} \psi d \mu - \pi^{\frac{n-1}{2}} e^{i \frac{\pi}{4} \omega} \sum_{\ell=0}^N \frac{i^\ell t^\ell}{2^{\ell} \ell!} \int  ((\square^*)^\ell \psi)(\gamma(s)) d \varphi(s) \label{curvelem} \end{equation}
is $O(t^{N+1})$, and specifically we have essentially the same expression for the error estimates, namely the inequality \eqref{error1} with the new operator $\square^*$, $n$ replaced by $n-1$ and $k > \frac{n-1}{2}$, i.e., \eqref{curvelem} is bounded in magnitude by a constant times 
\begin{equation} t^{N+1}  \left( \int | (\square^*)^{N+1} \psi| d \mu \right)^{1 - \frac{n-1}{2k}} \left( \int | (\square^*)^{N+k+1} \psi| d \mu \right)^{\frac{n-1}{2k}}. \label{errest2}
\end{equation}
Finally, since $\square$ (and hence $\square^*$) commute with multiplication by functions of $u$, we can multiply $\psi$ by $e^{i t^{-1} f(u)}$; the final conclusion is that
\begin{equation} t^{- \frac{n-1}{2}} \int e^{i t^{-1} \Phi} \psi d \mu - \pi^{\frac{n-1}{2}} e^{i \frac{\pi}{4} \omega} \sum_{\ell=0}^N \frac{i^\ell t^\ell}{2^{\ell} \ell!} \int_\gamma  e^{i t^{-1} \Phi} (\square^*)^\ell \psi ~ d \varphi \label{coareas}
\end{equation}
satisfies the error estimate \eqref{errest2}.

\subsection{One-dimensional analysis}

Before the proof of Theorem \ref{mainthm} can begin in full, there is a final issue to consider, namely the nature of the (now one-dimensional) phase $f(s) := \Phi (\gamma(s))$.   It is already established that $\nabla^2 \Phi - \frac{d \Phi}{du} \nabla^2 u$ is well-defined and independent of the connection along $\gamma$.  It is an easy calculation to see that
\begin{equation*} \nabla^2 (\Phi - f(u))(X,Y) = \nabla^2 \Phi(X,Y) - f'(u) \nabla^2 u(X,Y) - f''(u) (X u) (Y u).  \end{equation*}
Since $\Phi - f(u)$ is identically zero and has critical points along $\gamma$, we conclude that $\nabla^2 (\Phi - f(u)) (X, \dot \gamma) = 0$ along $\gamma$.  Thus
\begin{equation} f''(u) (X u) = \nabla^2 \Phi(X,\dot \gamma) - f'(u) \nabla^2 u(X,\dot \gamma) \label{secondd} \end{equation}
since we have parametrized $\gamma$ so that $\dot \gamma u = 1$.  
%
Let $X_1,\ldots,X_n$ be linearly-independent, smooth vector fields such that $X_j u = 0$ for $j=1,\ldots,n-1$ and $X_n = \dot \gamma$ on $\gamma$. If we look at the matrix $(\nabla^2 \Phi - f'(u) \nabla^2 u)(X_i,X_j)$, we see that the $ij$-entry is zero when $i = n$ and $j < n$ (or vice-versa). We also see that the $nn$-entry equals $f''(u)$.  Thus from \eqref{rnd} we find that
\begin{equation} f''(s) =  \pm \frac{\det (\nabla^2 \Phi - f'(s) \nabla^2 u)}{\mu^2}  \left( \frac{d \varphi}{ds} \right)^2
\label{seconder} \end{equation} 
(with the sign $\pm$ determined by the sign of the determinant of the minor $(\nabla^2 \Phi - f'(u) \nabla^2 u)(X_i,X_j)$ for $i,j \leq n-1$).

Now we return to the terms \eqref{coareas} appearing in the asymptotic expansion of
\[ t^{-\frac{n-1}{2}} \int e^{i t^{-1} \Phi} \eta(u) \psi ~ d \mu, \]
namely
\begin{equation} \int_I e^{i t^{-1} f(s)} \eta(s) \psi(\gamma(s)) \frac{d \varphi}{ds} ds \label{onedint2} \end{equation}
with $\psi$ replaced by $(\square^*)^\ell \psi$ for the higher-order terms, where we have introduced a smooth cutoff $\eta$ in the parameter $s$ (which equals $u$ in this parametrization).  For convenience, henceforth in this section we will let $\omega(s) := \eta(s) \psi(\gamma(s)) \frac{d \varphi}{ds}$. The next proposition shows how one can gain an additional factor of $t^{-1/2}$, bringing the total back to $t^{-n/2}$ as desired, when $u$ is equal to $(\det \nabla^2 \Phi / \mu^2)$ (i.e., when the cutoff function $u$ is chosen to be the Hessian determinant).

\begin{proposition}
Suppose $f : I \rightarrow \R$ is a $C^2$ function, and suppose that there exists a continuous function $g : I \rightarrow \R$ and constants $C$, $K$, and $\delta$ such that \label{estprop}
\begin{equation} |f''(s) - g(s)| \leq C |f'(s)|, \label{almost2} \end{equation}
\begin{equation} \delta \leq |g(s)| \leq K \delta. \label{comparable} \end{equation}
Then for all $t > 0$, 
\begin{equation} \left| \int_I e^{i t^{-1} f(s)} \omega(s) ds \right| \leq (t \delta^{-1})^{\frac{1}{2}} \left( 12 || \omega||_\infty + 4 || \omega'||_1 + 2 C \sqrt{K} || \omega||_1 \right). \label{uniform} \end{equation}
\end{proposition}
\begin{proof}
Consider the set $E_\epsilon := \set{ s \in I}{ |f'(s)| \leq \epsilon}$. This set must be connected when $\epsilon < C^{-1} \delta$. To see this, first observe that \eqref{almost2}, \eqref{comparable}, and the continuity of $g$ dictate that $f''(s) \neq 0$ on $E_\epsilon$ and has constant sign.  
Now choose any real number $a \in [-\epsilon,\epsilon]$ and consider the set of points $s$ such that $f'(s) = a$. If this set contains two points $s_1 < s_2$, then since $f''(s_1)$ and $f''(s_2)$ have the same sign, there must be an $s_3$ strictly in-between at which $f'(s_3) = a$ as well (thanks to the Intermediate Value Theorem). The point $s_3$ belongs to $E_\epsilon$, so iterating this process, we must be able to find a convergent sequence $\{s_k\}_{k=1}^\infty$ in $E_\epsilon$ such that $f'(s_k) = a$ for all $k$. But then the Mean Value Theorem and continuity of $f''$ require that $f''(\lim_{k} s_k) = 0$. Consequently the equation $f'(s) = a$ has at most one solution in $E_\epsilon$. But now if there happened to exist points $s < s' < s''$ such that $s, s'' \in E_\epsilon$ and $s' \not \in E_{\epsilon}$, then the Intermediate Value Theorem would dictate that $f'$ cannot possibly be single-valued on $[s,s''] \cap E_\epsilon$.

Returning to the proof of \eqref{uniform}, since
\[ \left| \int_I e^{i t^{-1} f(s)} \omega(s) ds \right| \leq || \omega||_1, \]
we may assume without loss of generality that $2 C \sqrt{ K t \delta^{-1}} \leq 1$.
If we fix $\epsilon = \sqrt{t \delta}$, then $\epsilon \leq \delta / (2C)$ (since $K \geq 1$). Thus it is established that $E_\epsilon$ is an interval and $|f''(s)| \geq  \delta/2$ on $E_\epsilon$. Consequently the length of $E_\epsilon$ is at most $2 \epsilon \delta^{-1}$, and
\[ \left| \int_{E_\epsilon} e^{i t^{-1} f(s)} \omega(s) ds \right| \leq 2 \epsilon \delta^{-1} ||\omega||_\infty. \]

For the remaining pieces, we divide $I \setminus E_\epsilon$ into sets $M_\epsilon$ and $H$, where $\epsilon < |f'(s)| < C^{-1} \delta$ on $M_\epsilon$ and $ |f'(s)| \geq  C^{-1} \delta$ on $H$. We know that $M_\epsilon$ is a union of no more than two intervals.  Consequently, $H$ can also be written as a union of no more than two intervals. On any interval $[a,b]$, the standard integration-by-parts trick found in the proof of van der Corput's lemma gives that
\[ \left| \int_a^b e^{i t^{-1} f(s)} \omega(s) ds \right| \leq t \! \left( \frac{|\omega(a)|}{|f'(a)|} \! + \! \frac{|\omega(b)|}{|f'(b)|} \! + \! \int_a^b \!  \left( |\omega(s)| \frac{|f''(s)|}{|f'(s)|^2} + \frac{|\omega'(s)|}{|f'(s)|} \right) ds \right). \]
On an interval of $M_\epsilon$, we have 
\[ \int_a^b \frac{|f''(s)|}{|f'(s)|^2}  ds = \left| \int_a^b \frac{f''(s)}{(f'(s))^2} ds \right| = \left| \frac{1}{f'(b)} - \frac{1}{f'(a)} \right|. \]
Thus
\[ \left| \int_a^b e^{i t^{-1} f(s)} \omega(s) ds \right| \leq 3 \epsilon^{-1} t  ||\omega||_\infty + \epsilon^{-1} t ||\omega'||_1 \]
on any such interval (with a coefficient of $3$ instead of $4$ because $f'(b)$ and $f'(a)$ have the same sign).
On an interval of $H$, on the other hand, we have $|f''(s)| \leq K \delta + C |f'(s)|$, so that
\[ \left| \int_a^b e^{i t^{-1} f(s)} \omega(s) ds \right| \leq 2 t \epsilon^{-1} ||\omega||_\infty +  C^2( K+1) t ||\omega||_1 \delta^{-1} + \epsilon^{-1} t || \omega'||_1 \]
(where we may use $\epsilon^{-1}$ in the first and last term since $\epsilon^{-1} < |f'(s)|$ here as well).
We conclude that
\[ \left| \int_I e^{i t^{-1} f(s)} \omega(s) ds \right| \leq  2 (\epsilon \delta^{-1} + 5 t \epsilon^{-1}) ||\omega||_\infty + 4 \epsilon^{-1} t ||\omega'||_1 + 2 C^2 (K+1) t \delta^{-1} ||\omega||_1. \]
Substituting $\epsilon = \sqrt{ t \delta}$ and using once again that $K \geq 1$ establishes the proposition since $2 C^2 (K+1) (t \delta^{-1}) \leq (2 C \sqrt{K t \delta^{-1}})^2$ and without loss of generality $(2 C \sqrt{K t \delta^{-1}})^2 \leq 2 C \sqrt{K t \delta^{-1}}$.
\end{proof}

\section{Proof of Theorem \ref{mainthm}}
\label{mainproof}

The proof of the main theorem proceeds by a standard bi-dyadic decomposition of the plane. On each piece we may assume that the amplitude $\psi$ is a smooth function of compact support. It will also be possible to assume that $\Phi$ is a smooth, real-valued phase defined on a neighborhood of the support of $\psi$ such that the mapping $x \mapsto \nabla \Phi(x)$ has at most Whitney fold singularities on this neighborhood.  The first step of the proof is to establish two inequalities for the integral on such pieces.  The first is that for any fixed $c > 0$ and any positive $N$,
\begin{equation} \left| \int e^{i \lambda (\Phi(x) + \xi \cdot x)} \chi( \epsilon^{-1} \det \hess \Phi(x)) \psi(x) dx \right| \lesssim \epsilon ( \lambda \epsilon)^{-N} \label{nocritical} \end{equation}
uniformly for all $\xi \in \R^2$ such that $|\nabla \Phi(x) + \xi| > c$ on the support of $\psi$ and uniformly in all positive $\epsilon, \lambda$, assuming that $\chi$ is smooth on $\R$ and compactly supported away from $0$.
The second is that
\begin{equation} \left| \int e^{i \lambda (\Phi(x) + \xi \cdot x)} \chi( \epsilon^{-1} \det \hess \Phi(x)) \psi(x) dx \right| \lesssim \lambda^{-1} \epsilon^{-\frac{1}{2}} \label{maybecritical} \end{equation}
uniformly for all $\xi \in \R^n$ and all $\lambda, \epsilon > 0$.

The proof of \eqref{nocritical} is a thoroughly standard ``non-stationary phase'' estimate, and follows, for example, from Lemma 2 of \cite{gressman2013III}. The key observation to be made in applying the lemma to the present situation is that the mapping $x \mapsto \nabla \Phi(x)$ having only Whitney folds means that the {\it gradient} of the Hessian determinant of $\Phi$ does not vanish when the Hessian determinant {\it does} vanish, so the zero set is a manifold, and the support of the integral is roughly contained in an $\epsilon$-neighborhood of the zero set of the Hessian determinant.

To prove \eqref{maybecritical}, it suffices to assume that $\xi$ lives in some compact subset of $\R^2$, since if $|\nabla \Phi + \xi| > c$ on the support of $\psi$, then the estimate \eqref{nocritical} applies and is sharper than the estimate \eqref{maybecritical} (since without loss of generality it may be assumed that $\epsilon$ is bounded above). Using compactness of the support of $\psi$ and the range of $\xi$ as well as a smooth partition of unity, we may assume without loss of generality that $\psi$ is supported within a small neighborhood of a critical point of $\Phi + x \cdot \xi$ for some $\xi$ under consideration (since every point $x$ which is not a critical point has a neighborhood, stable under small perturbations of $\xi$, on which \eqref{nocritical} holds uniformly for some small $c$).  In this case, $u = \det \hess \Phi$ will be a nondegenerate cutoff for the phase $\Phi(x) + \xi \cdot x$ by virtue of the assumption that $x \mapsto \nabla \Phi(x)$ has only Whitney folds.
We use \eqref{coareas} to estimate the integral.  In particular, we see that the difference
\begin{align*}
\int e^{i \lambda (\Phi(x) + \xi \cdot x)} & \chi (\epsilon^{-1} \det \hess \Phi(x)) \psi(x) dx \\ & - \pi^{\frac{n-1}{2}} \lambda^{-\frac{1}{2}} e^{i \frac{\pi}{4} \omega} \int e^{i \lambda( \Phi(\gamma(s)) + \gamma(s) \cdot \xi) }\chi( \epsilon^{-1} s) \psi(\gamma(s)) d \varphi(s) 
\end{align*}
will already be no greater than a constant times $\lambda^{-\frac{3}{2}}$, which again beats \eqref{maybecritical} since we may assume that $\lambda$ is bounded below in magnitude (otherwise the trivial estimate that the integral has size no greater than $\epsilon$ will beat \eqref{maybecritical} as well).  Now by \eqref{seconder} and Proposition \ref{estprop}, we conclude that
\[  \left| \int e^{i \lambda( \Phi(\gamma(s)) + \gamma(s) \cdot \xi) }\chi( \epsilon^{-1} s) \psi(\gamma(s)) d \varphi(s) \right| \lesssim \lambda^{-\frac{1}{2}} \epsilon^{-\frac{1}{2}}, \]
so \eqref{maybecritical} must hold as well.

The bi-dyadic decomposition itself is built from any smooth function $\eta$ supported on $[\frac{1}{2}, 2]$ such that 
\begin{equation} \sum_{j_1,j_2 \in \Z} \eta(2^{j_1} x_1, 2^{j_2} x_2) = 1 \label{partition} \end{equation}
away from the $x_1$- and $x_2$-axes.  For each $j := (j_1,j_2)$, we let
\[ \psi_{j}(x) := \psi(2^{-j_1} x_1, 2^{-j_2} x_2) \eta (|x_1|, |x_2|). \]
We note that for $j_1,j_2 \geq 0$, the $C^M$ norms of $\psi_{j}$ will be uniformly bounded for any fixed $M$, and if $\psi$ is supported sufficiently near the origin, we will have that
\[ \psi(x_1,x_2) = \sum_{j_1,j_2 \geq 0} \psi_{j} (2^{j_1} x_1, 2^{j_2} x_2) \]
away from the axes.

\subsection{Vertex Estimates}

Suppose $\alpha := (\alpha_1,\alpha_2)$ is a vertex of the Newton polygon of $\Phi$ with $\alpha_1, \alpha_2 \neq 0$.  We assume that $\alpha \neq (1,1)$, since if it were a vertex, the Hessian determinant of $\Phi$ would be nonvanishing on a neighborhood of the origin.
Let $j := (j_1,j_2)$ be any pair of positive integers such that 
\[ 2^{\alpha \cdot j} \Phi( 2^{-j_1} x_1, 2^{-j_2} x_2) - \frac{\partial^{\alpha} \Phi(0,0)}{\alpha!} x^{\alpha} \]
is a sufficiently small smooth function on the support of $\eta$ from \eqref{partition}.  Then
\begin{align*}
 \int e^{i \lambda (\Phi(x)+x \cdot \xi)} & \chi(\epsilon^{-1} u(x)) \psi_{j}(2^{j_1} x_1,2^{j_2} x_2) dx \\ =  2^{-\one \cdot j} & \int e^{i \lambda 2^{-\alpha \cdot j}  \Phi^\xi_j(x)} \chi(\epsilon^{-1} u(2^{-j_1}x_1, 2^{-j_2}x_2)) \psi_{j}(x) dx, \end{align*}
where
\[ \Phi_j^\xi(x) := 2^{\alpha \cdot j} \Phi(2^{-j_1} x_1,2^{-j_2} x_2) + 2^{\alpha \cdot j} (2^{-j_1} x_1 \xi_1 + 2^{-j_2} x_2 \xi_2). \] 
The Hessian determinant of $\Phi_j^\xi$ equals $2^{2( \alpha - \one) \cdot j} u(2^{-j_1} x_2,  2^{-j_2} x_2)$, and this must be a small perturbation of the Hessian determinant of the monomial term itself:
\[ \frac{\partial^{\alpha} \Phi(0,0)}{\alpha!} x^{\alpha}. \]
In particular, for sufficiently large $j$ (depending on $\partial^\alpha \Phi(0,0)$ and $\alpha$), it will be uniformly bounded above and below on the support of $\eta$.  Because $\chi$ is supported away from zero and infinity, the integrand will be identically zero unless $2^{-2(\alpha - \one)\cdot j} \approx \epsilon$.  On any such box, the derivatives of $\chi(\epsilon^{-1} u(2^{-j_1}x_1, 2^{-j_2}x_2))$ will be bounded independently of $\epsilon$ and $j$.

Putting these facts together with the estimate \eqref{maybecritical} (since the Hessian determinant of the {\it rescaled} phase is bounded above and below) gives
\[ \left|  \int e^{i \lambda \Phi^\xi(x)}  \chi(\epsilon^{-1} u(x)) \psi_{j}(2^{j_1}x_1,2^{j_2} x_2) dx \right| \lesssim 2^{-\one \cdot j} ( \lambda 2^{-\alpha \cdot j})^{-1} \approx \lambda^{-1} \epsilon^{-\frac{1}{2}} \]
uniformly in $\lambda$, $j$, and $\epsilon$.  Since the number of boxes on which $2^{-2(\alpha - \one)\cdot j} \approx \epsilon$ can hold is logarithmic in $\epsilon$ (since we have explicitly ruled out $\alpha = \one$), we can sum over all $j$ for only the cost of a logarithm of $\epsilon$.  If $\alpha$ is not on the diagonal, we can eliminate this logarithmic factor by observing that, for all but boundedly many of these boxes, the derivative of $\Phi_j^\epsilon$ must be bounded uniformly below.  Thus, we have the improved estimate \eqref{nocritical} on these boxes that
\[ \left|  \int e^{i \lambda \Phi^\xi(x)}  \chi(\epsilon^{-1} u(x)) \psi_j(2^{j_1}x_1,2^{j_2} x_2) dx \right| \lesssim 2^{-\one \cdot j} \min\{ 1, (\lambda 2^{- \alpha \cdot j})^{-N} \} \]
for any $N > 1$.  Summing over $j$'s such that $2^{-2(\alpha - \one) \cdot j} \approx \epsilon$ can be split into those terms on which $\lambda 2^{-\alpha \cdot j} \leq 1$ (on which the estimate $2^{-\one \cdot j}$ is used) and $\lambda 2^{-\alpha \cdot j} \geq 1$ (on which the estimate $2^{-\one \cdot j} (\lambda 2^{-\alpha \cdot j})^{-N}$ is used). As long as $\alpha$ is not on the diagonal, these estimates will decay exponentially away from the cut-off $\lambda 2^{-\alpha \cdot j} = 1$, and so the entire sum will converge and will be dominated by $2^{-\one \cdot j_0}$ for $j_0$ satisfying $\lambda = 2^{\alpha \cdot j_0}$ and $\epsilon = 2^{- 2(\alpha - \one) \cdot j_0}$, giving $2^{-\one \cdot j_0} \leq \lambda^{-1} \epsilon^{-1/2}$.

\subsection{Edge Estimates}

When considering the terms of the partition \eqref{partition}, the only boxes which are {\it not} covered by the so-called vertex estimates which were just proved are those boxes on which multiple boundary points of the Newton polygon of $\Phi$ (i.e., monomials in $x_1$ and $x_2$) have approximately the same size. These boxes are identifiable in terms of the compact edges of the polygon.  Generally, if the slopes of the compact faces are $-\beta_1^2,\ldots,-\beta_m^2$, then the only boxes from \eqref{partition} remaining to consider are boxes on which $\beta_\ell^{-1} j_1 - \beta_\ell j_2$ is uniformly bounded (for each of $\ell=1,\ldots,m$). There is, however an important exception: since no estimates were made for on-axis vertices, if edge $\ell$ meets the first (horizontal) axis then we must still deal with all boxes on which $\beta_\ell^{-1} j_1 - \beta_\ell j_2$ is bounded above. If the edge meets the second (vertical) axis, then we would still need to consider boxes on which $\beta_\ell^{-1} j_1 - \beta_\ell j_2$  is bounded below.  And if a single edge meets both axes, then no vertex estimates were proved and we would be starting from scratch.  In other words, we will fix a sufficiently large constant $C$ and for each integer $k \geq 0$ define the amplitude function
\[ \tilde \psi_k(2^{\beta_\ell k} x_1, 2^{\beta_\ell^{-1} k} x_2) := \mathop{\sum_{ |\beta_\ell^{-1} j_1 - k| \leq C}}_{|\beta_\ell j_2 - k| \leq C}  \psi_j(2^{j_1} x_1, 2^{j_2} x_2) \]
when edge $\ell$ meets neither axis, 
\[ \tilde \psi_k(2^{\beta_\ell k} x_1, 2^{\beta_\ell^{-1} k} x_2) := \mathop{\sum_{ |\beta_\ell^{-1} j_1 - k| \leq C}}_{\beta_\ell j_2 - k \geq -C}  \psi_j(2^{j_1} x_1, 2^{j_2} x_2) \]
when edge $\ell$ meets only the first (horizontal) axis, 
\[ \tilde \psi_k(2^{\beta_\ell k} x_1, 2^{\beta_\ell^{-1} k} x_2) := \mathop{\sum_{ \beta_\ell^{-1} j_1 - k \geq -C}}_{|\beta_\ell j_2 - k| \leq C}  \psi_j(2^{j_1} x_1, 2^{j_2} x_2) \]
when edge $\ell$ meets only the second (vertical) axis, and
\[ \tilde \psi_k(2^{\beta_\ell k} x_1, 2^{\beta_\ell^{-1} k} x_2) := \! \! \! \mathop{\sum_{ \beta_\ell^{-1} j_1 - k \geq -C}}_{|\beta_\ell j_2 - k| \leq C} \! \!   \psi_j(2^{j_1} x_1, 2^{j_2} x_2) +\! \! \!   \mathop{\sum_{ |\beta_\ell^{-1} j_1 - k| \leq C}}_{\beta_\ell j_2 - k > C} \! \!  \psi_j(2^{j_1} x_1, 2^{j_2} x_2)\]
when the edge $\ell$ meets both axes.  This way, the functions $\tilde \psi_k$ have uniformly bounded support and are always supported away from the origin.  They may or may not be supported away from the coordinate axes, but on their support we know that the mapping $(x_1, x_2) \mapsto \nabla p_\ell(x_1,x_2)$ (where $p_\ell$ is the polynomial corresponding to edge $\ell$) has at worst Whitney folds.  It is also true that the $C^M$-norms of these functions $\tilde \psi_k$ are uniformly bounded as a function of $k$.

In this case, we fix attention on the edge with slope $-\beta_\ell^2$. We pick $(j_1,j_2) = (\beta_\ell k, \beta_\ell^{-1} k)$ and make a similar scaling of $\Phi^\epsilon$ to the one used in the vertex estimates.  Let us specifically assume that the compact face, when extended, has first intercept $\beta_\ell^{-1} d_\ell (\beta_\ell + \beta_\ell^{-1})$ and second intercept $\beta_\ell d_\ell (\beta_\ell + \beta_\ell^{-1})$ (so that $d_\ell$ would equal the Newton distance of this particular extended face).  For convenience, let $\tilde \beta_\ell := \beta_\ell + \beta_\ell^{-1}$. Then we rescale by means of the transformation
\[ \Phi_k^\xi(x) := 2^{d_\ell \tilde \beta_\ell k} \Phi^\xi( 2^{-\beta_\ell k} x_1, 2^{-\beta_\ell^{-1} k} x_2). \]
Now for large $k$ we have that $\Phi_k^0 (x)$ differs from the polynomial $p_\ell(x)$, corresponding to the edge $\ell$ of the Newton polygon of $\Phi$, by a small smooth perturbation.  As before, the Hessian determinant of $\Phi_k^\xi$ is easily computed to equal $2^{2 (d_\ell-1) \tilde \beta_\ell k } u( 2^{-\beta_\ell k} x_1, 2^{-\beta_\ell^{-1} k} x_2)$.  We may assume that the distance $d_\ell$ is greater than one (since, by taking $k \rightarrow \infty$, this could only happen when the Hessian determinant of $\Phi$ was nonvanishing at the origin).  So the cutoff $u \approx \epsilon$ corresponds to the Hessian determinant of $\Phi_k$ being approximately $\epsilon 2^{2 (d_\ell - 1) \tilde \beta_\ell k}$. In particular, the number of such terms $k$ can therefore be at most comparable to the logarithm of $\epsilon$ (since the Hessian determinant of $\Phi_k$ will be uniformly bounded above on the support of a cutoff function $\eta$ which is contained in $[-2,2] \times [-2,2]$ in the rescaled coordinates).

By \eqref{maybecritical}, we have
\begin{align*}
 \left| 2^{- \tilde \beta_\ell k} \int \right. & \left. \vphantom{\int} e^{i \lambda 2^{-d_\ell \tilde \beta_\ell k} \Phi_k^\xi(x)} \chi(\epsilon^{-1} 2^{-2 (d_\ell - 1) \tilde \beta_\ell k} \det \hess \Phi_k(x) ) \tilde \psi_k (x)  dx \right| \\
 & \lesssim 2^{-\tilde \beta_\ell k} (\lambda 2^{-d_\ell \tilde \beta_\ell k})^{-1} (\epsilon^{-1} 2^{-2 (d_\ell - 1) \tilde \beta_\ell k})^{\frac{1}{2}} \approx \lambda^{-1} \epsilon^{-\frac{1}{2}}.
 \end{align*}
 Simply summing this estimate over $k$ gives the desired estimate $\lambda^{-1} \epsilon^{-\frac{1}{2}}$ times a logarithmic factor in $\epsilon$. However, as before, we may eliminate the logarithm by using the fact that, for any fixed $\xi$, there will be at most a bounded number of boxes on which the gradient of $\Phi^\xi_k$ vanishes. Away from this finite collection of boxes, we can use \eqref{nocritical} with $N=1$ to obtain
\begin{align*}
 \left| 2^{- \tilde \beta_\ell k} \int \right. & \left. \vphantom{\int} e^{i \lambda 2^{-d_\ell \tilde \beta_\ell k} \Phi_k^\xi(x)} \chi(\epsilon^{-1} 2^{-2 (d_\ell - 1) \tilde \beta_\ell k} \det \hess \Phi_k(x) ) \tilde \psi_k (x) dx \right| \\
 & \lesssim 2^{-\tilde \beta_\ell k} (\lambda 2^{-d_\ell \tilde \beta_\ell k})^{-1} (\epsilon 2^{2 (d_\ell - 1) \tilde \beta_\ell k})^{1-1} \approx 2^{(d_\ell - 1) \tilde \beta_\ell k} \lambda^{-1}.
 \end{align*}
Since we have assumed $d_\ell > 1$, when we sum over $k$, the entire sum will be comparable to the term for the largest value of $k$, which occurs when $2^{2(d_\ell - 1) \tilde \beta_\ell k} \approx \epsilon^{-1}$. So once again, the entire sum is dominated by $\lambda^{-1} \epsilon^{-1/2}$.

\bibliography{mybib}
\end{document}